\documentclass[12pt,a4paper]{amsart}
\usepackage{amsmath, amssymb, amsthm, eufrak, setspace, verbatim}
\usepackage[top=2.5cm, bottom=2cm, left=2cm, right=2cm]{geometry}
\usepackage[all]{xy}
\newtheorem{thm}{Theorem}[section]
\newtheorem{prop}[thm]{Proposition}

\newtheorem{cor}[thm]{Corollary}
\theoremstyle{definition}

\newtheorem{example}[thm]{Example}

\newcommand{\Q}{\mathbb{Q}}
\newcommand{\M}{{\mathfrak M}}
\newcommand{\A}{{\mathfrak A}}
\newcommand{\p}{{\mathfrak p}}
\newcommand{\B}{{\mathfrak B}}
\newcommand{\Bp}{{\mathfrak B}_{\p}}
\newcommand{\AH}{{\mathfrak A}_{H}}
\newcommand{\OK}{{\mathfrak O}_{K}}
\newcommand{\OKp}{{\mathfrak O}_{K,{\p}}}
\newcommand{\OL}{{\mathfrak O}_{L}}
\newcommand{\OLO}{{\mathfrak O}_{L_{0}}}
\newcommand{\OLp}{{\mathfrak O}_{L,{\p}}}
\newcommand{\OF}{{\mathfrak O}_{F}}

\renewcommand{\OE}{{\mathfrak O}_{E}}
\newcommand{\perm}[1]{\mbox{Perm}(#1)}
\newcommand{\Gal}[1]{\mbox{Gal}(#1)}

\newcommand{\ndivides}{ \hspace{0mm} \nmid }
\renewcommand{\bar}{\overline}

\title{Commutative Hopf-Galois module structure of tame extensions}
\author{Paul J. Truman}
\address{School of Computing and Mathematics, Keele University, Staffordshire, ST5 5BG, UK}
\email{P.J.Truman@Keele.ac.uk}


\begin{document}

\maketitle

\begin{abstract}
We prove three theorems concerning the Hopf-Galois module structure of fractional ideals in a finite tamely ramified extension  of $ p $-adic fields or number fields which is $ H $-Galois for a commutative Hopf algebra $ H $. Firstly, we show that if $ L/K $ is a tame Galois extension of $ p $-adic fields then each fractional ideal of $ L $ is free over its associated order in $ H $. We also show that this conclusion remains valid if $ L/K $ is merely almost classically Galois. Finally, we show that if $ L/K $ is an abelian extension of number fields then every ambiguous fractional ideal of $ L $ is locally free over its associated order in $ H $. 
\end{abstract}

\section{Introduction} \label{section_introduction}

Hopf-Galois module theory is a generalization of the classical Galois module theory of algebraic integers. Classically, one considers a finite Galois extension of local or global fields $ L/K $ with group $ G $; the field $ L $ is then a free module of rank one over the group algebra $ K[G] $, and for each fractional ideal $ \B $ of $ L $ we seek criteria for $ \B $ to be free (or locally free) over its associated order in $ K[G] $:
\[ \A_{K[G]}(\B) = \{ z \in K[G] \mid z \cdot x \in \B \mbox{ for all } x \in \B \}. \]
The group algebra $ K[G] $ is a Hopf algebra, and its action on the field $ L $ is an example of a {\em Hopf-Galois structure} on the extension $ L/K $ (the formal definition follows in section \ref{section_HGS}). However, there may be other Hopf algebras that give Hopf-Galois structures on $ L/K $, and each fractional ideal $ \B $ of $ L $ has an associated order in each of these, defined analogously to $ \A_{K[G]}(\B) $ above. This raises the possibility of comparing the structure of $ \B $ as a module over its associated orders in the various Hopf algebras giving Hopf-Galois structures on the extension. Furthermore, a finite separable, but non-normal, extension of local or global fields may admit Hopf-Galois structures; here of course the techniques of classical Galois module theory are not available, but these Hopf-Galois structures allow us to study such extensions and their fractional ideals as described above. For a thorough survey of the applications of Hopf algebras to questions of local Galois module theory, we refer the reader to \cite{ChTwe}. 
\\ \\
These techniques have proven to be fruitful in the study of wildly ramified Galois extensions, where determining the structure of a fractional ideal $ \B $ over its associated order $ \A_{K[G]}(\B) $ is a difficult problem. The most striking results in this direction are due to Byott \cite{By97a}, who exhibited a class of wildly ramified Galois extensions of $ p $-adic fields $ L/K $ for which the valuation ring $ \OL $ is not free over $ \A_{K[G]}(\OL) $, but is free over $ \AH(\OL) $ for some other Hopf algebra $ H $ giving a Hopf-Galois structure on the extension. Thus for these extensions this Hopf-Galois structure is ``better" than that given by $ K[G] $ for the purposes of describing $ \OL $. 
\\ \\ \\
If $ L/K $ is a Galois extension of $ p $-adic fields which is at most tamely ramified (henceforth, simply ``tame") then by Noether's Theorem \cite[Theorem 3]{Fro} $ \OL $ is a free $ \OK[G] $-module, and so there is no possibility of a different Hopf-Galois structure giving a ``better" description of $ \OL $. However, in \cite{PJT_Towards} and \cite{PJT_Integral}, we found that for many such extensions there is remarkable uniformity in the descriptions of $ \OL $ afforded by the various Hopf-Galois structures. For example, if $ L/K $ is a Galois extension of $ p $-adic fields and $ H $ is a Hopf algebra giving a Hopf-Galois structure on $ L/K $ then $ \OL $ is a free $ \AH(\OL) $-module if the extension is unramified \cite[Theorem 3.4]{PJT_Towards}, or if $ H $ is commutative and $ p \ndivides [L:K] $ \cite[Theorem 4.4]{PJT_Towards}. The main theorem of the present paper (Theorem \ref{thm_main})
is a considerable generalization of the latter result: we prove that if $ L/K $ is a tame Galois extension of $ p $-adic fields, $ H $ a commutative Hopf algebra giving a Hopf-Galois structure on $ L/K $, and $ \B $ a fractional ideal of $ L $, then $ \B $ is a free $ \AH(\B) $-module. The results of sections \ref{section_action_of_G0} and \ref{section_induced} culminate in a proof of this theorem. In section \ref{section_non_normal} we prove an analogue of Theorem \ref{thm_main} for non-normal extensions of $ p $-adic fields that are {\em almost classically Galois} (a separable extension $ L/K $ with Galois closure $ E/K $ is almost classically Galois if $ \Gal{E/L} $ has a normal complement in $ \Gal{E/K} $ \cite[Definition 4.2]{GreitherPareigis}). Finally, in section \ref{section_number_fields} we prove an analogue of Theorem \ref{thm_main} for abelian extensions of number fields. 

\section{Hopf-Galois structures and Greither-Pareigis theory} \label{section_HGS}

Let $ K $ be a field and $ H $ a $ K $-Hopf algebra with conuit $ \varepsilon : H \rightarrow K $ and comultiplication $ \Delta: H \rightarrow H \otimes_{K} H $. For $ h \in H $, write $ \Delta(h) = \sum_{(h)} h_{(1)} \otimes h_{(2)} \in H \otimes_{K} H $ (Sweedler's notation). We say that an extension $ L $ of $ K $ is an {\em $ H $-module algebra} if $ L $ is an $ H $-module, $ h \cdot 1 = \varepsilon (h)1 $, and for all $ h \in H $ and $ s,t \in L $ we have $ h \cdot(st) = \sum_{(h)} (h_{(1)} \cdot s) (h_{(2)} \cdot t) $. We say that $ L $ is an {\em $ H $-Galois extension of $ K $}, or that $ H $ gives a {\em Hopf-Galois structure} on $ L/K $, if $ L $ is an $ H $-module algebra and the $ K $-linear map
\[ j: L \otimes_{K} H \rightarrow \mbox{End}_{K}(L) \]
defined by
\[ j(s \otimes h)(t) = s(h \cdot t) \mbox{ for all } s,t \in L, \; h \in H \]
is an isomorphism of $ K $-vector spaces. The prototypical example of a Hopf-Galois structure on a finite extension of fields is the action of the group algebra $ K[G] $ on a finite Galois extension $ L/K $ with group $ G $. 
\\ \\ 
In the case that $ L/K $ is separable, a theorem of Greither and Pareigis (\cite[Theorem 3.1]{GreitherPareigis} or \cite[Theorem 6.8]{ChTwe}) implies that each Hopf algebra giving a Hopf-Galois structure on $ L/K $ has the form $ E[N]^{G} $, where $ E/K $ is the normal closure of $ L/K $, $ G=\Gal{E/K} $, $ N $ is a regular subgroup of the permutation group of the set $ X = \Gal{E/K} / \Gal{E/L} $ that is normalized by the image of the left translation map $ \lambda : G \rightarrow \perm{X} $, and $ G $ acts on $ E[N] $ by acting on $ E $ as Galois automorphisms and on $ N $ by conjugation via $ \lambda $ . Such a Hopf algebra acts on $ L $ by
\begin{equation}\label{eqn_GP_action} 
\left( \sum_{\eta \in N} c_{\eta} \eta \right) \cdot x = \sum_{\eta \in N} c_{\eta} \eta^{-1}(\bar{1})[x] \hspace{5mm} (c_{\eta} \in E, \; x \in L).
\end{equation}
If $ L/K $ is a Galois extension then $ X = G $ and $ \lambda $ is the left regular representation of $ G $. In this case one example of a regular subgroup of $ \perm{G} $ normalized by $ \lambda(G) $ is $ \rho(G) $, the image of $ G $ under the right regular representation. In fact, since $ \lambda(G) $ centralizes $ \rho(G) $ we have $ L[\rho(G)]^{G} = L^{G}[\rho(G)] = K[\rho(G)] $, and the action expressed in equation \eqref{eqn_GP_action}  reduces to the usual action of $ K[G] $ on $ L $. We call this the {\em classical Hopf-Galois structure} on $ L/K $. We call any other Hopf-Galois structures admitted by a Galois extension, and all those admitted by a non-normal extension, {\em nonclassical}. 
\\ \\
We will view Hopf algebras arising from the theorem of Greither and Pareigis as distinct if they differ as sets, and always think of such a Hopf algebra $ H=E[N]^{G} $ as acting via equation \eqref{eqn_GP_action}; we may therefore speak of {\em the} Hopf-Galois structure given by $ H $. We shall say that $ N $ is the {\em underlying group} of $ H $, and refer to the isomorphism class of $ N $ as the {\em type} of $ H $. The integral group ring $ \Lambda = \OE[N] $ is an $ \OE $-order in the group algebra $ E[N] $, and $ \Lambda^{G} = \OE[N]^{G} $ is an $ \OK $-order in the Hopf algebra $ E[N]^{G} $. It is a natural analogue of the integral group ring $ \OK[G] $ within a group algebra $ K[G] $ acting on a Galois extension $ L/K $ with group $ G $. In particular, we have:

\begin{prop} \label{prop_Lambda_G_subset_A_H}
If $ \B $ is a fractional ideal of $ L $ and $ \B = \B^{\prime} \cap L $ for some ambiguous fractional ideal $ \B^{\prime} $ of $ E $, then $ \Lambda^{G} \subseteq \AH(\B) $.
\end{prop}
\begin{proof}
This is a minor generalization of \cite[Proposition 2.5]{PJT_Towards}. Let $ x \in \B $ and $ z \in \Lambda^{G} $. Then clearly $ z \cdot x \in L $. On the other hand, we may write $ z =  \sum_{n \in N} c_{n} n $ with $ c_{n} \in \OE $, and then by equation \eqref{eqn_GP_action} we have
\[ z \cdot x = \sum_{n \in N} c_{n} n^{-1}(\bar{1})[x]. \]
For each $ n \in N $, any group element representing $ \eta^{-1}(\bar{1}) $ is a Galois automorphism of $ E $, so (since $ \B^{\prime} $ is an ambiguous fractional ideal) $ n^{-1}(\bar{1})[x] \in \B^{\prime} $ for each $ n \in N $. Therefore $ z \cdot x \in \B^{\prime} $, and so in fact $ z \cdot x \in \B^{\prime} \cap L = \B $, whence $ z \in \AH(\B) $. 
\end{proof}

We note that if $ L/K $ is an extension of local fields then every fractional ideal of $ E $ is ambiguous, and so $ \Lambda^{G} \subseteq \AH(\B) $ for every fractional ideal of $ L $. In \cite{PJT_Towards} and \cite{PJT_Integral} we found that for many tame extensions $ L/K $ we actually have $ \AH(\OL)=\Lambda^{G} $, providing a nice analogue of the fact that if $ L/K $ is Galois with group $ G $ then $ \A_{K[G]}(\OL) = \OK[G] $ by Noether's Theorem. We will continue to focus on $ \Lambda^{G} $ in the present paper. 

\section{The action of inertia} \label{section_action_of_G0}

Let $ L/K $ be a tame Galois extension of $ p $-adic fields with group $ G $, let $ \B $ be a fractional ideal of $ L $, and let $ H $ be a Hopf algebra giving a Hopf-Galois structure on $ L/K $. By the theorem of Greither and Pareigis $ H =  L[N]^{G} $ for $ N $ some regular subgroup of $ \perm{G} $ normalized by $ \lambda(G) $. In \cite[Theorems 1.1 and 1.2]{PJT_Integral} we showed that if the inertia subgroup $ G_{0} $ of $ G $ acts trivially on $ N $  then $ \OL $ is a free $ \Lambda^{G} $-module. In this section we will assume that $ H $ is commutative (i.e. $ N $ is abelian), and prove the following variant of this result:

\begin{thm} \label{thm_G0_trivial_p_part}
Suppose that $ N $ is abelian and that $ G_{0} $ acts trivially on the $ p $-part of $ N $. Then $ \B $ is a free $ \Lambda^{G} $-module. 
\end{thm}
 
Since $ \AH(\B) $ is the only order in $ H $ over which $ \B $ can possibly be free \cite[Proposition 12.5]{ChTwe}, Theorem \ref{thm_G0_trivial_p_part} implies that $ \AH(\B) = \Lambda^{G} $ in this case. In section \ref{section_induced} we will use the theory of induced Hopf-Galois structures to show that in fact $ G_{0} $ always acts trivially on the $ p $-part of $ N $. Our first step in proving Theorem \ref{thm_G0_trivial_p_part} is the following:

\begin{prop} \label{prop_free_iff_projective}
$ \B $ is a free $\Lambda^{G} $-module if and only if it is a projective $ \Lambda^{G} $-module. 
\end{prop}
\begin{proof}
If $ \B $ is a free $ \Lambda^{G} $-module then it is a projective $ \Lambda^{G} $-module. Conversely, suppose that $ \B $ is a projective $ \Lambda^{G} $-module, and observe that since $ \Lambda^{G} $ is an order in a commutative separable algebra over the $ p $-adic field $ K $, it is a clean order \cite[IX, Corollary 1]{Roggenkamp1970_II}. This means that the projective $ \Lambda^{G} $-module $ \B $ is free if and only if $ K \otimes_{\OK} \B $ is a free $ K \otimes_{\OK} \Lambda^{G} $-module \cite[IX, Theorem 1.2]{Roggenkamp1970_II}. We have $ K \otimes_{\OK} \B = L $ and  $ K \otimes_{\OK} \Lambda^{G} = H $, and by the generalization of the normal basis theorem to Hopf-Galois structures \cite[Theorem 2.16]{ChTwe} $ L $  is indeed a free $ H $-module. Therefore $ \B $ is a free $ \Lambda^{G} $-module. 
\end{proof}

Next we study the role of the inertia subgroup $ G_{0} $. To do this, we examine in a little more detail the process by which the group $ N $ yields a Hopf-Galois structure on $ L/K $, following the proof of the theorem of  Greither and Pareigis, as presented in \cite[\S 6]{ChTwe}. Since $ N $ is a regular subgroup of $ \perm{G} $ the group algebra $ L[N] $ gives a Hopf-Galois structure on the extension of rings $ M/L $, where  $ M = \mbox{Map}(G,L) $. To describe the action of $ N $ on $ M $, for each $ g \in G $ let $ u_{g} \in M $ be defined by $ u_{g}(h) = \delta_{g,h} $ for all $ h \in G $. Then $ \{ u_{g} \mid g \in G \} $ is a basis of orthogonal idempotents for $ M $, and $ N $ acts by permuting the subscripts: 
\[ \eta \cdot u_{g} = u_{\eta(g)} \mbox{ for } \eta \in N \mbox{ and } g \in G. \]
Since in addition $ N $ is normalized by $ \lambda(G) $, the group $ G $ acts on $ L[N] $ by acting on $ L $ as Galois automorphisms and on $ N $ by conjugation via the image of the embedding $ \lambda $ into $ \perm{G} $. It also acts on $ M $ by acting on $ L $ as Galois automorphisms and on the idempotents $ u_{g} $ by left translation of the subscripts. The action of $ L[N] $ on $ M $ is $ G $-equivariant with respect to these actions, and so by Galois descent we obtain that the $ K $-Hopf algebra $ L[N]^{G} $ gives a Hopf-Galois structure on the extension of rings $ M^{G} / K $. Finally, we may identify $ L $ with the fixed ring $ M^{G} $ via the $ K $-algebra isomorphism $ L \xrightarrow{\sim} M^{G} $ defined by
\begin{equation} \label{eqn_GP_isomorphism}
 x \mapsto f_{x} =  \sum_{g \in G} g(x) u_{g} \mbox{ for all } x \in L. 
 \end{equation}
Thus, the Hopf-Galois structure given by $ L[N]^{G} $ on $ M^{G}/K $ translates to one on $ L/K $, with the action of $ L[N]^{G} $ on $ L $ as given in equation \eqref{eqn_GP_action}. 
\\ \\
At integral level, the group ring $ \Lambda = \OL[N] $ acts on the maximal order $ \mbox{Map}(G,\OL) $ and so, if $ G_{F} $ is a subgroup of $ G $ with fixed field $ F $, the $ \OF $-algebra $ \Lambda^{G_{F}} $ acts on the $ \OF $-algebra $  \mbox{Map}(G,\OL)^{G_{F}} $. In particular, $ \Lambda^{G} $ acts on the $ \OK $-algebra $ \mbox{Map}(G,\OL)^{G} $, which is isomorphic to $ \OL $ as an $ \OK $-algebra and a $ \Lambda^{G} $-module.  Similarly, let $ \Gamma_{\B} = \mbox{Map}(G,\B) $; then $ \Gamma_{\B}^{G} $ is isomorphic to $ \B $ as a $ \Lambda^{G} $-module. For the rest of this section we study the structure of $ \Gamma_{\B}^{G} $ as a $ \Lambda^{G} $-module. Since the fractional ideal $ \B $ will be fixed throughout, we will usually suppress the subscript $ \B $ and write simply $ \Gamma $. With this notation, we have:

\begin{prop} \label{prop_G_to_G0}
$ \Gamma^{G} $ is a projective $ \Lambda^{G} $-module if and only if $ \Gamma^{G_{0}} $ is a projective $ \Lambda^{G_{0}} $-module.
\end{prop}
\begin{proof}
Let $ L_{0} $ denote the fixed field of $ G_{0} $, so that $ L_{0}/K $ is the maximal unramified subextension of $ L/K $, and let $ \bar{G} = G/G_{0}  \cong \Gal{L_{0}/K} $. Then we have $ \Lambda^{G} = \left( \Lambda^{G_{0}} \right)^{\bar{G}} $ and $ \Gamma^{G} = \left( \Gamma^{G_{0}} \right)^{\bar{G}} $ (with $ \bar{G} $ acting in the obvious way). Since $ L_{0}/K $ is unramified, the extension of rings $ \OLO / \OK $ is a Galois extension, and so by Galois descent the inverse to the fixed module functor (sending an $ \OLO $-module $ M $ to the $ \OK $-module $ M^{\bar{G}} $) is the base change functor (sending an $ \OK $-module $ M^{\prime} $ to the $ \OLO $-module $  \OLO \otimes_{\OK} M^{\prime} $). Therefore we have
\[ \OLO \otimes_{\OK} \left( \Lambda^{G_{0}} \right)^{\bar{G}} = \Lambda^{G_{0}} \mbox{ and } \OLO \otimes_{\OK} \left( \Gamma^{G_{0}} \right)^{\bar{G}} = \Gamma^{G_{0}}. \]
That is:
\[ \OLO \otimes_{\OK} \Lambda^{G} = \Lambda^{G_{0}} \mbox{ and } \OLO \otimes_{\OK} \Gamma^{G} = \Gamma^{G_{0}}. \]
Now by \cite[Lemma 5.1]{Frohlich_Invariants} $ \Gamma^{G} $ is a projective $ \Lambda^{G} $-module if and only if $ \OLO \otimes_{\OK} \Gamma^{G} $ is a projective $ \OLO \otimes_{\OK} \Lambda^{G} $-module. That is, if and only if $ \Gamma^{G_{0}} $ is a projective $ \Lambda^{G_{0}} $-module.
\end{proof}

In order to complete the proof of Theorem \ref{thm_G0_trivial_p_part}, it remains to show that if the action of $ G_{0} $ on the $ p $-part of $ N $ is trivial then $ \Gamma^{G_{0}} $ is a projective $ \Lambda^{G_{0}} $-module. We will accomplish this by showing that $ \Lambda^{G_{0}} $ takes a particular form in this case. Writing $ T $ for the $ p $-part of $ N $ and $ S $ for the prime-to-$p$-part of $ N $, we have:

\begin{prop} \label{prop_LambdaG0_group_ring}
Suppose that the action of $ G_{0} $ on $ T $ is trivial. Then 
\[ \Lambda^{G_{0}} = \left(\OL[S]^{G_{0}}\right)[T]. \]
\end{prop}
\begin{proof}
First note that the subgroups $ S,T $ are $ G $-stable, since for all $ g \in G $ and $ \eta \in N $ the permutations $ \eta $ and $ \lambda(g) \eta \lambda(g^{-1}) $ have the same order. We have $ N = S \times T  $ and $ \OL[N] = \OL[S][T] $, and since $ G_{0} $ acts trivially on $ T $ we have 
\[ \Lambda^{G_{0}} = \OL[N]^{G_{0}} = \OL[S][T]^{G_{0}} = \left(\OL[S]^{G_{0}}\right)[T]. \]
\end{proof}
Thus if $ G_{0} $ acts trivially on $ T $ then  $ \Lambda^{G_{0}} $ is a group ring of the finite $ p $-group $ T $ with coefficients in the ring $ \OL[S]^{G_{0}} $. This description of $ \Lambda^{G_{0}} $ will allow us to imitate the classical proof of the fact that $ \B $ is a projective $ \OK[G] $-module (see for example \cite[Proposition 1.3]{Ullom69}). In the classical situation, $ \OK $ is the maximal order in $ K $, so $ \B $ is certainly a projective (in fact, a free) $ \OK $-module. Therefore, if $ X $ is a free $ \OK[G] $ module and $ \varphi : X \rightarrow \B $ a surjective $ \OK[G] $-homomorphism, then there certainly exists an $ \OK $-homomorphism $ \psi_{0} : \B \rightarrow X $ that splits $ \varphi $. Since $ L/K $ is tame, there exists an element $ x \in \OL $ such that $ \mathrm{Tr}_{L/K}(x)=1 $, and multiplication by $ x $ is an $ \OK $-linear endomorphism of $ \B $. Therefore the map $ \psi : \B \rightarrow X $ defined by
\[ \psi(y) = \sum_{g \in G} g \psi_{0}(x  g^{-1}(y) ) \mbox{ for all } y \in \B \]
is an $ \OK[G] $-homomorphism, and splits $ \varphi $, which implies that $ \B $ is a projective $ \OK[G] $-module. 
\\ \\
In our imitation, $ \OL[S]^{G_{0}} $ will play a role analogous to that played by $ \OK $ in the classical proof:

\begin{prop} \label{prop_OLSG0_maximal}
$ \OL[S]^{G_{0}} $ is the unique maximal order in the group algebra $ L[S]^{G_{0}} $.
\end{prop}
\begin{proof}
The argument of \cite[Proposition 4.2]{PJT_Towards} essentially applies. The group algebra $ L[S] $ is separable and commutative, and therefore contains a unique maximal $ \OL $-order, which is equal to $ \OL[S] $ since $ p \ndivides |S| $ \cite[Proposition 27.1]{MORT_I}. Similarly, the $ L_{0} $-algebra $ L[S]^{G_{0}} $ contains a unique maximal $ \OLO $-order $ \M  $, which is the integral closure of $ \OLO $ in $ L[S]^{G_{0}} $. Clearly $ \OL[S]^{G_{0}}\subseteq \M $, but if $ z \in \M $ then $ z $ is integral over $ \OLO $, hence integral over $ \OL $, so $ z \in \OL[S] $. In addition, $ z $ is fixed by every element of $ G_{0} $, so in fact $ z \in \OL[S]^{G_{0}} $. Therefore $ \M = \OL[S]^{G_{0}} $.
\end{proof}

Since $ \OL[S]^{G_{0}} $ is a maximal order, it is hereditary, and so $ \Gamma^{G_{0}} $ is certainly a projective $ \OL[S]^{G_{0}} $-module. Therefore, if $ X $ is a free $ \Lambda^{G_{0}} $-module and $ \varphi : X \rightarrow \Gamma^{G_{0}} $ is a surjective $ \Lambda^{G_{0}} $ homomorphism, there exists an $ \OL[S]^{G_{0}} $-homomorphism $ \psi_{0} : \Gamma^{G_{0}} \rightarrow X $ that splits $ \varphi $. As in the proof of the classical result, we will construct a $ \Lambda^{G_{0}} $-homomorphism $ \psi : \Gamma^{G_{0}} \rightarrow X $ that splits $ \varphi $. To do this, we require an element $ f \in \mbox{Map}(G,\OL)^{G_{0}} $ analogous to the element $ x \in \OL $ of trace $ 1 $ in the proof of the classical result. Let
\[ \theta_{N} = \sum_{\eta \in N} \eta, \hspace{2mm}  \theta_{S} = \sum_{\sigma \in S} \sigma,   \mbox{ and } \theta_{T} = \sum_{\tau \in T} \tau. \]
It is clear that $ \theta_{N}, \theta_{S}, \theta_{T} \in \Lambda^{G} \subseteq \Lambda^{G_{0}} $ and that $ \theta_{N} = \theta_{S}\theta_{T} $.

\begin{prop} \label{prop_element_f_analogous_to_x}
There exists $ f \in  \mbox{Map}(G,\OL)^{G_{0}} $ such that $ \theta_{T} \cdot f = 1 $. 
\end{prop}
\begin{proof}
Since $ L/K $ is tame, there exists $ x \in \OL $ such that $ \mathrm{Tr}_{L/K}(x)=1 $, and this element corresponds under the isomorphism in \eqref{eqn_GP_isomorphism} to the element 
\[ f_{x} = \sum_{g \in G} g(x) u_{g} \in  \mbox{Map}(G,\OL)^{G} \subseteq  \mbox{Map}(G,\OL)^{G_{0}}. \]
Note that
\[ \theta_{N} \cdot f_{x} = \sum_{\eta \in N} \eta  \cdot \left(\sum_{g \in G} g(x) u_{g}\right) = \sum_{\eta \in N} \sum_{g \in G} g(x) u_{\eta(g)} = \sum_{g \in G} \left( \sum_{\eta \in N} \eta^{-1}(g)(x) \right) u_{g} = \sum_{g \in G} u_{g} = 1, \]
since $ N $ is regular on $ G $ and $ \mathrm{Tr}_{L/K}(x)=1 $. Now let $ f = \theta_{S}  \cdot f_{x} \in \mbox{Map}(G,\OL)^{G_{0}} $. Then we have
\[ \theta_{T} \cdot  f = \theta_{T} \cdot (\theta_{S} \cdot f_{x}) = (\theta_{T}\theta_{S}) \cdot f_{x} = \theta_{N}  \cdot f_{x} = 1. \]
\end{proof}

We can verify that $ f $ does possess the properties we require for our imitation of the classical proof:

\begin{prop} \label{prop_mult_by_f}
Multiplication by $ f $ is an $ \OL[S]^{G_{0}} $-endomorphism of $ \Gamma^{G_{0}} $.
\end{prop}
\begin{proof}
Multiplication by $ f $ is an endomorphism of $ \Gamma^{G_{0}} $ because $ \Gamma^{G_{0}} $ is an ideal of $ \mbox{Map}(G,\OL)^{G_{0}} $. To show that it is $ \OL[S]^{G_{0}} $-linear, it is sufficient to show that multiplication by $ f $ is an $ \OL[S] $-endomorphism of $ \Gamma $. Let $ z \in \OL[S] $, and write $ z = \sum_{\sigma \in S} c_{\sigma} \sigma $ with $ c_{\sigma} \in \OL $. Note that $ \Delta(z) = \sum_{\sigma \in S} c_{\sigma} \sigma \otimes \sigma $, and that since $ \theta_{S} $ is an integral of the group ring $ \OL[S] $ we have
\[ z \cdot f = z \cdot ( \theta_{S} \cdot f_{x} ) = (z \theta_{S}) \cdot  f_{x} = (\varepsilon(z) \theta_{S}) \cdot  f_{x} = \varepsilon(z) f, \]
where $ \varepsilon $ denotes the counit map of $ L[S] $; in particular, $ \sigma \cdot  f = f $ for all $ \sigma \in S $. Now for all $ \gamma \in \Gamma $ we have:
\begin{eqnarray*}
z \cdot (f \gamma) & = & \sum_{\sigma \in S} c_{\sigma} (\sigma \cdot  f) (\sigma \cdot \gamma) \mbox{ since } M \mbox{ is an } L[S] \mbox{-module algebra} \\
& = & \sum_{\sigma \in S} c_{\sigma} f (\sigma  \cdot \gamma) \mbox{ since } \sigma  \cdot f = f \mbox{ for all } \sigma \in S \\
& = & f \sum_{\sigma \in S} c_{\sigma} (\sigma \cdot \gamma) \\
& = & f (z \cdot \gamma).
\end{eqnarray*}
Therefore multiplication by $ f $ is an $ \OL[S] $-endomorphism of $ \Gamma $, and so multiplication by $ f $ is an $ \OL[S]^{G_{0}} $-endomorphism of $ \Gamma^{G_{0}} $.
\end{proof}

Hence, assuming the action of $ G_{0} $ on $ T $ is trivial, $ f $ can be used to extend an $ \OL[S]^{G_{0}} $-homomorphism to a $ \Lambda^{G_{0}} $-homomorphism:

\begin{prop} \label{prop_Lambda_G_0_hom}
Suppose that the action of $ G_{0} $ on $ T $ is trivial. Let $ X $ be a $ \Lambda^{G_{0}} $-module and $ \psi_{0} : \Gamma^{G_{0}} \rightarrow X $ an $ \OL[S]^{G_{0}} $-homomorphism. Define $ \psi : \Gamma^{G_{0}} \rightarrow X $ by
\[ \psi(\gamma) = \sum_{\tau \in T} \tau \cdot \psi_{0}( f \tau^{-1} \cdot \gamma ). \]
Then $ \psi $ is a $ \Lambda^{G_{0}} $-homomorphism.
\end{prop}
\begin{proof}
Since the action of $ G_{0} $ on $ T $ is trivial, we have $ \Lambda^{G_{0}} = \OL[S]^{G_{0}}[T] $ by Proposition \ref{prop_LambdaG0_group_ring}. It is therefore is sufficient to show that 
\[ \psi( (z \upsilon) \cdot \gamma ) = (z \upsilon) \cdot  \psi(\gamma) \mbox{ for all } z \in \OL[S]^{G_{0}}, \hspace{2mm} \upsilon \in T, \hspace{2mm} \gamma \in \Gamma^{G_{0}}. \]
For such $ z, \upsilon $ and $ \gamma $, we have
\begin{eqnarray*}
\psi( (z \upsilon)  \cdot \gamma ) & = & \sum_{\tau \in T} \tau \cdot \psi_{0}( f \tau^{-1} (z \upsilon) \cdot \gamma ) \\
& = & \sum_{\tau \in T} z \tau \cdot \psi_{0}( f \tau^{-1} \upsilon \cdot \gamma ) \\
& = & \sum_{\tau \in T} z \upsilon \tau \cdot \psi_{0}( f \tau^{-1} \cdot \gamma ) \mbox{ (replacing $\tau $ with $ \upsilon^{-1}\tau $)} \\
& = & (z\upsilon) \cdot \psi(\gamma).
\end{eqnarray*}
(The second equality holds because multiplication by $ f $ is $ \OL[S]^{G_{0}} $-linear (by Proposition \ref{prop_mult_by_f}), and $ \psi_{0} $ is an $ \OL[S]^{G_{0}} $-homomorphism.) Thus  $ \psi $ is a $ \Lambda^{G_{0}} $-homomorphism.
\end{proof}

Now, completing our imitation of the classical proof, we have:

\begin{prop}\label{prop_GammaG0_proj}
Suppose that the action of $ G_{0} $ on $ T $ is trivial. Then $ \Gamma^{G_{0}} $ is a projective $ \Lambda^{G_{0}} $-module. 
\end{prop}
\begin{proof}
Let $ X $ be a free $ \Lambda^{G_{0}} $-module, and $ \varphi : X \rightarrow \Gamma^{G_{0}} $ a surjective $ \Lambda^{G_{0}} $-homomorphism. By Proposition \ref{prop_OLSG0_maximal} $ \OL[S]^{G_{0}} $ is the unique maximal $ \OLO $-order in $ L[S]^{G_{0}} $, and so there exists an $ \OL[S]^{G_{0}} $-homomorphism $ \psi_{0} : \Gamma^{G_{0}} \rightarrow X $ that splits $ \varphi $. Now define $ \psi : \Gamma^{G_{0}} \rightarrow X $ by
\[ \psi(\gamma) = \sum_{\tau \in T} \tau\cdot  \psi_{0}( f \tau^{-1} \cdot \gamma ). \]
By Proposition \eqref{prop_Lambda_G_0_hom}, $ \psi $ is a $ \Lambda^{G_{0}} $-homomorphism, and for all $ \gamma \in \Gamma^{G_{0}} $ we have
\begin{eqnarray*}
\varphi \psi ( \gamma ) & = & \sum_{\tau \in T} \varphi(\tau \cdot \psi(f \tau^{-1} \cdot \gamma)) \\
& = & \sum_{\tau \in T} \tau \cdot \varphi(\psi(f \tau^{-1} \cdot \gamma)) \\
& = & \sum_{\tau \in T} \tau \cdot ( f \tau^{-1} \cdot \gamma)) \\
& = & \sum_{\tau \in T} (\tau \cdot f) \gamma \\
& = & \left( \theta_{T}  f \right) \cdot \gamma \\
& = & \gamma, 
\end{eqnarray*}
by Proposition \ref{prop_element_f_analogous_to_x}. Therefore $ \psi $ is a $ \Lambda^{G_{0}} $-homomorphism splitting $ \varphi $, and so $ \Gamma^{G_{0}} $ is a projective $ \Lambda^{G_{0}} $ module.
\end{proof}

Combining the results of this section, we now have:

\begin{proof}[Proof of Theorem \ref{thm_G0_trivial_p_part}.]
By Proposition \ref{prop_free_iff_projective}, it is sufficient to show that $ \B $ is a projective $ \Lambda^{G} $-module and, since $ \B \cong \Gamma^{G} $ as $ \Lambda^{G} $-modules, this is equivalent to showing that $ \Gamma^{G} $ is a projective $ \Lambda^{G} $-module. By Proposition \ref{prop_G_to_G0} this is equivalent to showing that $ \Gamma^{G_{0}} $ is a projective $ \Lambda^{G_{0}} $-module, and by Proposition \ref{prop_GammaG0_proj} this is true. Therefore $ \B $ is a free $ \Lambda^{G} $-module. 
\end{proof}

\section{Induced Hopf-Galois structures on tame extensions} \label{section_induced}

In \cite[Theorem 3]{Crespo_Induced} Crespo, Rio and Vela show that if  $ L/K $ is a Galois extension of fields with group $ G $, $ F/K $ is a subextension, and $ G_{F} =  \Gal{L/F} $ has a normal complement in $ G $ then, given Hopf-Galois structures on $ F/K $ and $ L/F $ with underlying groups $ S, T $ respectively, we can induce a Hopf-Galois structure on $ L/K $ whose underlying group is $ S \times T $. Conversely, they consider a Hopf-Galois structure on $ L/K $ whose underlying group $ N $ is the direct product of two $ G $-stable subgroups  $ S, T $. In this situation $ L[T]^{G} $ is a $ K $-Hopf subalgebra of $ H $, and we may consider its fixed field:
\[ L^{T} = \{ x \in L \mid z \cdot x = \varepsilon(z)x \mbox{ for all } z \in L[T]^{G} \}. \]
As in classical Galois theory we have $ [L:L^{T}] = |T| $.  Under the assumption that $ \Gal{L/L^{T}} $ has a normal complement in $ G $, \cite[Theorem 9]{Crespo_Induced} asserts that there are Hopf-Galois structures on $ L^{T}/K $ and $ L/L^{T} $ having underlying groups $ S $ and $ T $ respectively, and that the original Hopf-Galois structure is induced from these. In this section we shall show that every commutative Hopf-Galois structure on a tame Galois extension of $ p $-adic fields is induced, and deduce that in this situation the hypotheses of Theorem \ref{thm_G0_trivial_p_part} are always satisfied. 
\\ \\
From the proof of \cite[Theorem 3]{Crespo_Induced} we can extract the following:

\begin{prop} \label{prop_induced_complement_trivial_action}
Let $ L/K $ be a Galois extension of fields and $ F/K $ a subextension such that $ G_{F} = \Gal{L/F} $  has a normal complement $ C $ in $ G $. Suppose that there are Hopf-Galois structures on $ F/K $ and $ L/F $, with underlying groups $ S $ and $ T $ respectively, so that the Hopf-Galois structure on $ L/K $ induced by these has underlying group $ N =S \times T $. Then the action of $ C $ on $ T $ is trivial. 
\end{prop}
\begin{proof}
Let  $ C = \{ x_{1}, \ldots ,x_{s} \} $ and $ G_{F} = \{ y_{1}, \ldots ,y_{t} \} $, so that $ G = \{ x_{i}y_{j} \mid 1 \leq i \leq s, \; 1 \leq j \leq t \} $. Following \cite{Crespo_Induced}, we identify $ \perm{C} $ and $ \perm{G_{F}} $ with the symmetric groups $ \mathrm{Symm}(s) $ and $ \mathrm{Symm}(t) $ respectively via their action on subscripts. There is an injective homomorphism
\[ \iota : \mathrm{Symm}(s) \times \mathrm{Symm}(t) \rightarrow \mathrm{Symm}(st) = \mathrm{Symm}( \{ 1, \ldots ,s\} \times \{ 1, \ldots ,t \}) \]
defined by
\[ \iota ( \sigma , \tau ) ( i,j ) = ( \sigma(i), \tau(j) ). \]
Next we describe the action of $ G $ on itself  via the left regular embedding $ \lambda_{G} : G \rightarrow \mathrm{Symm}(st) $. We have the left regular embedding $ \lambda_{T} : G_{F} \rightarrow \mathrm{Symm}(t) $. In \cite[Theorem 3]{Crespo_Induced} a homomorphism $ \lambda_{S} : G \rightarrow \mathrm{Symm}(s) $ is constructed, and it is shown that the action of $ G $ on itself by left translation is then given by
\[ \lambda_{G}(xy) (x_{i}y_{j}) = x_{\lambda_{S}(xy)[i]} y_{\lambda_{T}(y)[j]}. \]
Now consider the group $ T $ underlying the Hopf-Galois structure on $ L/F $. It is a regular subgroup of $ \perm{G_{F}} $, which we are identifying with $ \mathrm{Symm}(t) $ via its action on the subscripts of the elements $ y_{j} $. To show that $ C $ acts trivially on $ T $ we must show that $ \lambda_{G}(x)\tau \lambda_{G}(x)^{-1} = \tau $ for all $ x \in C $. For an arbitrary element $ x_{i}y_{j} \in G $, we have:
\begin{eqnarray*}
\lambda_{G}(x)\tau \lambda_{G}(x^{-1}) [x_{i}y_{j}] & = & \lambda_{G}(x)\tau x_{\lambda_{S}(x^{-1})[i]} y_{j} \\
& = & \lambda_{G}(x) x_{\lambda_{S}(x^{-1})[i]} y_{\tau(j)} \\
& = & x_{ \lambda_{S}(x)(\lambda_{S}(x^{-1})[i])} y_{\tau(j)} \\
& = & x_{i} y_{\tau(j)} \\
& = & \tau [x_{i}y_{j}].
\end{eqnarray*}
Thus the action of $ C $ on $ T $ is trivial.
\end{proof}

We note that this situation is not symmetric: the action of $ G_{F} $ on $ S $ is not trivial in general. 

\begin{example}
Let $ L/K $ be a Galois extension of fields with Galois group 
\[ G = \langle a,b \mid a^{3}=b^{2}=1, bab=a^{-1} \rangle \cong D_{3}, \]
and let $ F = L^{\langle b \rangle} $. Then $ G_{F} = \langle b \rangle $ has a normal complement $ C = \langle a \rangle $ in $ G $. Identifying $ \perm{C} $ with $ \perm{G/G_{F}} $, let $ \sigma \in \perm{C} $ and $ \tau \in \perm{G_{F}} $ be defined by $ \sigma(a^{i}) = a^{i+1} $ and $ \tau(b^{j})=b^{j+1} $, and let $ S = \langle \sigma \rangle $ and $ T = \langle \tau \rangle $. It is easy to verify that $ S,T $ correspond via the theorem of Greither and Pareigis to Hopf-Galois structures on $ F/K $, $ L/F $ respectively. Since $ G_{F} $ has a normal complement in $ G $, we may induce from these a Hopf-Galois structure on $ L/K $, whose underlying group $ S \times T $ acts on $ G $ as follows:
\[ (\sigma^{u}, \tau^{v}) [a^{i}b^{j}] = a^{i+u}b^{j+v}. \]
In this case $ G_{F} $ does not act trivially on $ S $. We have:
\begin{eqnarray*}
\lambda_{G}(b) \sigma \lambda_{G}(b^{-1}) [a^{i}b^{j}] & = & \lambda_{G}(b) \sigma [a^{-i}b^{j-1}] \\
& = & \lambda_{G}(b) [a^{1-i}b^{j-1}] \\
& = & a^{i-1}b^{j} \\
& = & \sigma^{-1} [a^{i} b^{j}]. 
\end{eqnarray*}
Thus $ \lambda_{G}(b) \sigma \lambda_{G}(b^{-1}) = \sigma^{-1} $, and so the action of $ G_{F} $ on $ S $ is not trivial. 
\end{example}

We now specialize to the situation considered in section \ref{section_action_of_G0}:  $ L/K $ is a tame Galois extension of $ p $-adic fields with group $ G $ and inertia subgroup $ G_{0} $, $ H = L[N]^{G} $ is a commutative Hopf algebra giving a Hopf-Galois structure on the extension, $ T $ is the $ p $-part of $ N $, and $ S $ is the prime-to-$p$ part of $ N $, so that $ N = S \times T $. We shall show that the Hopf-Galois structure given by $ H $ is induced from Hopf-Galois structures on $ L^{T}/K $ and $ L/L^{T} $. 

\begin{prop} \label{prop_tame_normal_complement}
Let $ F = L^{T} $ and $ G_{F} = \Gal{L/F} $. Then the subgroup $ G_{F}  $ has a normal complement in $ G $ containing $ G_{0} $. 
\end{prop}
\begin{proof}
Let $ e= |G_{0}| $ denote the ramification index, and $ f = |G|/|G_{0}| $ the residue field degree, of $ L/K $, so that $ |G| = ef $, and let $ p^{r} $ be the largest power of $ p $ that divides $ |G| $. Since $ L/K $ is tame, $ e $ is coprime to $ p $,  and we may write $ f = p^{r}f_{0} $ with $ f_{0} $ coprime to $ p $. 
Note that the quotient group $ G / G_{0} $ is cyclic of order $ f $, and let $ C \subset G $ be the kernel of the composition of homomorphisms
\[ \begin{array}{lccccc}
& G & \rightarrow & G/ G_{0} & \rightarrow & G/G_{0} \\
\mbox{defined by } &g & \mapsto & gG_{0} & \mapsto & g^{f_{0}}G_{0}.
\end{array} \]
Then $ C $ is a normal subgroup of $ G $ which contains $ G_{0} $ and has order $ ef_{0} $ and index $ p^{r} $. Since these are coprime, by the Schur-Zassenhaus Theorem \cite[Theorem 8.35]{MORT_I} $ C $ has a complement in $ G $, and all of the complements of $ C $ in $ G $ are conjugate. But any complement of $ C $ in $ G $ must be a Sylow $ p $-subgroup of $ G $, and these are all conjugate. Therefore the distinct complements of $ C $ in $ G $ are precisely the Sylow $ p $-subgroups of $ G $. But $ |G_{F}| = [L:L^{T}] = |T| = p^{r} $ since $ T $ is the unique Sylow $p$-subgroup of $ N $, and so $ G_{F} $ is a Sylow $p$-subgroup of $ G $. Therefore $ G_{F} $ is a complement of $ C $ in $ G $, and so $ C $ is a normal complement of $ G_{F} $ in $ G $ containing $ G_{0} $. 
\end{proof}

\begin{prop} \label{prop_tame_structures_induced}
The Hopf-Galois structure on $ L/K $ with underlying group $ N $ is induced from Hopf-Galois structures on $ L^{T}/K $ and $ L/L^{T} $ with underlying groups $ S $ and $ T $ respectively. 
\end{prop}
\begin{proof}
By Proposition \ref{prop_LambdaG0_group_ring}, $ S $ and $ T $ are $ G $-invariant subgroups of $ N $ such that $ N = S \times T $, and by Proposition \ref{prop_tame_normal_complement} $ \Gal{L/L^{T}} $ has a normal complement in $ G $. Therefore we may apply \cite[Theorem 9]{Crespo_Induced} and conclude that there is a Hopf-Galois structure on $ L^{T}/K $ with underlying group $ S $ a Hopf-Galois structure on $ L/L^{T} $ with underlying group $ T $, and that the Hopf-Galois structure on $ L/K $ with underlying group $ N $ is induced from these. 
\end{proof}

\begin{cor} \label{cor_G0_T_trivial}
The action of $ G_{0} $ on $ T $ is trivial. 
\end{cor}
\begin{proof}
By Proposition \ref{prop_induced_complement_trivial_action}, the normal complement $ C $ of $ \Gal{L/L^{T}} $ in $ G $ acts trivially on $ T $, and by Proposition \ref{prop_tame_normal_complement} $ G_{0} \subseteq C $. 
\end{proof}

We can now state and prove our main theorem:

\begin{thm} \label{thm_main}
Let $ L/K $ be a tame Galois extension of $ p $-adic fields, $ H = L[N]^{G} $ a commutative Hopf algebra giving a Hopf-Galois structure on $ L/K $, and $ \B $ a fractional ideal of $ L $. Then $ \B $ is a free $ \OL[N]^{G} $-module.
\end{thm}
\begin{proof}
By corollary \ref{cor_G0_T_trivial} the action of the inertia group $ G_{0} $ on the $ p $-part of $ N $ is trivial, and by Theorem \ref{thm_G0_trivial_p_part} this implies that $ \B $ is a free $ \OL[N]^{G} $-module. 
\end{proof}

The assumption that the group $ N $ underlying $ H $ is abelian has been crucial to many of our results, in particular Propositions \ref{prop_LambdaG0_group_ring} and \ref{prop_OLSG0_maximal}. One would not, therefore, expect the arguments presented thus far to generalize to the noncommutative case. In fact, we can give an example to show that the direct generalization of Theorem \ref{thm_main} to noncommutative Hopf-Galois structures does not hold:

\begin{example}
Let $ p $ be a prime number that is congruent to $ 2 $ modulo $ 3 $, so that $ K = \Q_{p} $ does not contain a primitive cube root of unity, and let $ L $ be the splitting field of $ x^{3}-p $ over $ K $.  Then $ L = K(a, \zeta) $, where $ a^{3}=p $ and $ \zeta $ is a primitive cube root of unity, and $ L/K $ is Galois with group $ G \cong D_{3} $. Since $ G $ is nonabelian, $ L/K $ admits a distinguished nonclassical Hopf-Galois structure, with underlying group $ \lambda(G) $. The corresponding Hopf algebra is $ H_{\lambda} = L[\lambda(G)]^{G} $, and we shall write $ \A_{\lambda} $ for the associated order of $ \OL $ in $ H_{\lambda} $. We shall show that $ \OL[\lambda(G)]^{G} \subsetneq \A_{\lambda} $, which implies that $ \OL $ is not a free $ \OL[\lambda(G)]^{G} $-module. 
\\ \\
We present $ G $ as 
\[ G = \langle \sigma, \tau \mid \sigma^{3} = \tau^{2} = 1, \tau \sigma \tau = \sigma^{-1}  \rangle, \]
where 
\[ \begin{array}{ll}
\sigma(a) = \zeta a, & \sigma(\zeta) = \zeta \\
\tau(a)=a, & \tau(\zeta)=\zeta^{-1}. \end{array} \]
The extension $ L/K $ cannot be unramified, since it has nonabelian Galois group, and it cannot be totally ramified, since the subextension $ K(\zeta)/K $ is nontrivial and unramified. Therefore the inertia subgroup must be the unique nontrivial proper normal subgroup of $ G $, which is $ \langle \sigma \rangle $, and so $ L/K $ is tamely ramified. Recall from the theorem of Greither and Pareigis that $ G $ acts on $ L[\lambda(G)] $ by acting on $ L $ as Galois automorphisms and acting on $ \lambda(G) $ by conjugation via the left regular embedding $ \lambda : G \rightarrow \perm{G} $. We may verify that the element 
\[ z = a^{2} \lambda(\tau) + \zeta^{2} a^{2} \lambda(\sigma^{-1} \tau) + \zeta a^{2} \lambda(\sigma \tau) \in \OL[\lambda(G)] \]
is fixed by every element of $ G $, and so actually lies in $ \OL[\lambda(G)]^{G} $. Therefore it acts on elements of $ L $ via equation \eqref{eqn_GP_action}. It is sufficient to consider its action on elements of the form $ \zeta^{i}a^{j} $ for $ i=1,2 $ and $ j=0,1,2 $, since these form an $ \OK $-basis of $ \OL $. For such an element, we have:
\begin{eqnarray*}
z \cdot  (\zeta^{i}a^{j}) & = & a^{2}(\lambda(\tau) + \zeta^{2} \lambda(\sigma^{-1} \tau) + \zeta \lambda(\sigma \tau) ) \cdot (\zeta^{i} a^{j}) \\
& = & a^{2}(\lambda(\tau)^{-1}[1_{G}] + \zeta^{2} \lambda(\sigma^{-1} \tau)^{-1}[1_{G}] + \zeta \lambda(\sigma \tau)^{-1}[1_{G}] ) (\zeta^{i} a^{j}) \\
& = & a^{2} ( \tau + \zeta^{2} \tau \sigma + \zeta \tau \sigma^{2})(\zeta^{i} a^{j}) \\
& = & a^{2+j} \zeta^{-i} (1 + \zeta^{2(1+j)} + \zeta^{(1+j)} ) \\
& = & \left\{ \begin{array}{cl} 3p\zeta^{-i}a & \mbox{if } j=2 \\ 0 & \mbox{otherwise} \end{array} \right. \\
& \in & p \OL 
\end{eqnarray*}
Hence $ z \cdot  x \in p \OL $ for all $ x \in \OL $, and so $ z / p  $ lies in $ \A_{\lambda} $ but not in $ \OL[\lambda(G)]^{G} $. In fact, it can be shown that $ {\displaystyle \A_{\lambda} =  \OL[\lambda(G)]^{G}\left[ z/p \right] } $, and that $ \OL $ is a free $ \A_{\lambda} $-module. (One could also deduce freeness of $ \OL $ over $ \A_{\lambda} $ from \cite[Theorem 1.1]{PJT_Canonical}, without constructing $ \A_{\lambda} $ explicitly.)
\end{example}

Thus, if $ L/K $ is a tame Galois extension of $ p $-adic fields with group $ G $ and $ H = L[N]^{G} $ is a noncommutative Hopf algebra giving a Hopf-Galois structure on the extension, then in general  $ \AH(\OL) $ may strictly contain $ \OL[N]^{G} $, and when this occurs $ \OL $ cannot be a free $ \OL[N]^{G} $-module. However, we are not aware of an example of a tame Galois (nor, indeed, of a tame, separable but non-normal) extension of local fields $ L/K $ admitting a Hopf-Galois structure $ H $ for which $ \OL $ is not free over $ \AH(\OL) $. 

\section{Non-normal extensions} \label{section_non_normal}

As mentioned in section \ref{section_introduction}, one of the ways in which Hopf-Galois module theory extends classical Galois module theory is that a separable, but non-normal, extension of local or global fields may admit Hopf-Galois structures, which can then be used to study the fractional ideals in these extensions. In this section we prove  an analogue of Theorem \ref{thm_main} for almost classically Galois extensions (see section \ref{section_introduction} for the definition). Once again our strategy employs induced Hopf-Galois structures: we will show that, given a tame almost classically Galois extension of $ p $-adic fields $ L/K $ which is Hopf-Galois for a commutative Hopf algebra, it is possible to induce a Hopf-Galois structure on the Galois closure $ E/K $ and  deduce information about the structure of fractional ideals of $ L $ from the structure of those of $ E $. Our first result formulates this deduction precisely; it is an analogue of \cite[Lemma 6]{Byott_Lettl} for induced Hopf-Galois structures, and applies more generally than our current situation. 

\begin{prop} \label{prop_non_normal_descent}
Let $ E/K $ be a Galois extension of local or global fields with group $ G $, and let $ L/K $ be a subextension such that $ E/L $ is at most tamely ramified. Suppose that $ G_{L} = \Gal{E/L} $ has a normal complement in $ G $, that there are Hopf-Galois structures on $ L/K $ and $ E/L $ given by Hopf algebras $ H_{S}, H_{T} $ with underlying groups $ S,T $ respectively, and that $ H $, with underlying group $ N = S \times T $, gives the Hopf-Galois structure on $ E/K $ induced by these. Let $ \B^{\prime} $ be an ambiguous fractional ideal of $ E $, and suppose that $ \B^{\prime} $ is free over its associated order in $ H $. Then $ \B = \B^{\prime} \cap L $ is free over its associated order in $ H_{S} $. 
\end{prop}
\begin{proof}
By the theorem of Greither and Pareigis we have $ H_{S} = \tilde{L}[S]^{\tilde{G}} $, where $ \tilde{L}/K $ is the Galois closure of $ L/K $ and $ \tilde{G} = \Gal{\tilde{L}/K} $. We may express this as the fixed ring of the group algebra $ E[S] $ under a certain action of $ G $. 
Since $ \tilde{L} / K $ is Galois, $ G^{\prime} = \Gal{E/\tilde{L}} $ is a normal subgroup of $ G $, and $ G / G^{\prime} \cong \tilde{G} $. Let $ G $ act on $ E[S] $ by acting on $ E $ as Galois automorphisms and on $ S $ by factoring through $ \tilde{G} $. Then we have
\[ E[S]^{G} = \left( E[S]^{G^{\prime}} \right)^{\tilde{G}} = \left( E^{G^{\prime}}[S] \right)^{\tilde{G}} = \left( \tilde{L}[S] \right)^{\tilde{G}} = H_{S}. \]
Now let $ \pi : E[N] \rightarrow E[S] $ be the $ E $-algebra homomorphism induced by the projection $ N \rightarrow S $, so that
\[ \pi \left( \sum_{\sigma \in S, \tau \in T} c_{\sigma,\tau} \sigma\tau \right) =  \sum_{\sigma \in S, \tau \in T} c_{\sigma,\tau} \sigma \mbox{ for all } c_{\sigma,\tau} \in E. \]
Since $ S,T $ are $ G $-stable subgroups of $ N $, $ \pi $ is $ G $-equivariant, where the $ G $ acts on $ E[N] $ as stated in the theorem of Greither and Pareigis, and on $ E[S] $ as defined above. It therefore restricts to a $ K $-algebra homomorphism $ \pi : H \rightarrow H_{S} $. Now let 
\[ \theta_{T} = \sum_{\tau \in T} \tau \in H_{T}. \]
For all $ x \in E $, we have
\[ \theta_{T} \cdot  x = \left(\sum_{\tau \in T} \tau \right) \cdot x = \sum_{\tau \in T} \tau^{-1}(1_{G})[x] = \sum_{g \in {\scriptstyle \mathrm{Gal}(E/L)}} g(x) = \mathrm{Tr}_{E/L}(x). \]
In addition, note that $ \theta_{T} \tau = \theta_{T} $ for all $ \tau \in T $, and that $ \sigma \theta_{T} = \theta_{T} \sigma $ for all $ \sigma \in S $. Therefore, if
\[ z = \sum_{\sigma \in S, \tau \in T} c_{\sigma,\tau} \sigma\tau \in H, \] 
then for all $ x \in E $ we have
\begin{eqnarray*}
( z \theta_{T} ) \cdot  x & = & \left( \sum_{\sigma \in S, \tau \in T} c_{\sigma,\tau} \sigma \tau \theta_{T} \right) \cdot  x \\
& = & \left( \sum_{\sigma \in S, \tau \in T} c_{\sigma,\tau} \sigma \theta_{T} \right)  \cdot x \\
& = & \left( \sum_{\sigma \in S, \tau \in T} c_{\sigma,\tau} \sigma  \right) \cdot \left(\theta_{T} \cdot x \right)\\
& = & \pi(z) \cdot \mathrm{Tr}_{E/L}(x).
\end{eqnarray*}
Now let $ \A = \AH(\B^{\prime}) $, and suppose that $ \B^{\prime} $ is a free $ \A $-module, say $ \B^{\prime} =\A \cdot  x $. Since $ E/L $ is at most tamely ramified and $ \B^{\prime} $ is an ambiguous fractional ideal of $ E $, we have $\B = \B^{\prime} \cap L = \mathrm{Tr}_{E/L}(\B^{\prime}) $ \cite[Corollary 1.2]{Ullom69}, and so:
\[ \B = \mathrm{Tr}_{E/L}(\B^{\prime}) = \theta_{T} \cdot  \B^{\prime} = \theta_{T}  ( \A \cdot x ) = (\A  \theta_{T}) \cdot  x = \pi(\A) \cdot \mathrm{Tr}_{E/L}(x). \]
Thus $ \B $ is a free $ \pi(\A) $-module of rank $ 1 $, and so $ \pi(\A) $ is the associated order of $ \B $ in $ H_{S} $. 
\end{proof}

This proposition facilitates an approach to studying the Hopf-Galois module structure of fractional ideals in an almost classically Galois extension $ L/K $ of local or global fields $ L/K $ whose Galois closure $ E/K $ has the property that $ E/L $ is at most tamely ramified. Given a Hopf-Galois structure on $ L/K $, we choose a Hopf-Galois structure admitted by $ E/L $ (as a Galois extension, it will always admit at least the classical Hopf-Galois structure), and use the fact that $ \Gal{E/L} $ has a normal complement in $ \Gal{E/K} $ to induce from these a Hopf-Galois structure on the Galois extension $ E/K $ (see the discussion of induced Hopf-Galois structures at the start of section \ref{section_induced}). Now applying Proposition \ref{prop_non_normal_descent} allows us to deduce information about the structure of fractional ideals of $ L $ with respect to the given Hopf-Galois structure from information about the structure of fractional ideals of $ E $ with respect to the induced Hopf-Galois structure. 
\\ \\
In order to apply this approach to a tame almost classically Galois extension of $ p $-adic fields, we verify that the Galois closure of a tame extension of $ p $-adic fields is again tame:

\begin{prop} \label{prop_galois_closure_tame_extension}
Let $ L/K $ be a tame non-normal extension of $ p $-adic fields with Galois closure $ E/K $. Then $ E/L $ is unramified and $ E/K $ is tame. 
\end{prop}
\begin{proof}
Let $ e $ denote the ramification index, and $ L_{0}/K $ the maximal unramified subextension, of $ L/K $, so that $ L/L_{0} $ is totally ramified of degree $ e $. There exists a prime element $ \pi_{L} $ of $ L $ such that $ \pi_{L}^{e} = v\pi_{K} $ with $ v $  a unit of $ \OLO $. Let $ \zeta $ be a primitive $ e^{th} $ root of unity, $ u $ an $ e^{th} $ root of $ v $, and $ L^{\prime} = L(\zeta,u) $. Then $ L^{\prime}/L $ is unramified, since $ e $ is prime to $ p $. Let $ \pi_{L}^{\prime} = u^{-1}\pi_{L} \in L^{\prime} $. Then $ (\pi_{L}^{\prime})^{e} = \pi_{K} $, so $ \pi_{L}^{\prime} $ is a root of the polynomial $ x^{e}-\pi_{K} $, which is defined over $ K $ and splits over $ L^{\prime} $. Now let $ L_{0}^{\prime}/K $ denote the maximal unramified subextension of $ L^{\prime}/K $, and let $ f(x) \in K[x] $ be a polynomial such that $ L_{0}^{\prime} $ is the splitting field of $ f(x) $ over $ K $. Then the smallest extension of $ K $ over which both $ x^{e}-\pi_{K} $ and $ f(x) $ split is a Galois extension of $ K $ containing $ L $ and contained in $ L^{\prime} $, and so the Galois closure $ E/K $ of $ L/K $ is contained in $ L^{\prime} $, which is an unramified extension of $ L $. Therefore $ E/L $ is unramified and the ramification index of $ E/K $ is $ e $, so $ E/K $ is tame. 
%\[
%\xymatrixcolsep{6pc} 
%\xymatrixrowsep{2pc}
%\xymatrix{
%& L^{\prime} \ar@{-}[d]  \ar@{-}[ddr] & \\
%& E \ar@{-}[dd] & \\
%&& L_{0}^{\prime} \ar@{-}[dddl] \\
%& L \ar@{-}[d] & \\
%& L_{0} \ar@{-}[d]   & \\
%& K \ar@{-}[d] & \\
%& \Q_{p} & 
%} 
%\]
\end{proof}

Now we have:

\begin{thm} \label{thm_non_normal}
Let $ L/K $ be a tame almost classically Galois extension of $ p $-adic fields whose Galois closure $ E/K $ has group $ G $. Let $ H = L[N]^{G} $ be a commutative Hopf algebra giving a Hopf-Galois structure on $ L/K $, and let $ \B $ be a fractional ideal of $ L $. Then $ \B $ is a free $ \OL[N]^{G} $-module.
\end{thm}
\begin{proof}
By Proposition \ref{prop_galois_closure_tame_extension}, the extension $ E/L $ is unramified, hence cyclic; let $ G_{L} $ denote its Galois group. Under the correspondence established by the theorem of Greither and Pareigis, the classical Hopf-Galois structure on $ E/L $ corresponds to the regular subgroup $ \rho(G_{L}) $ of $ \perm{G_{L}} $. Since $ L/K $ is an almost classically Galois extension, $ G_{L} $ has a normal complement in $ G $, and so by \cite[Theorem 3]{Crespo_Induced} we may induce a Hopf-Galois structure on $ E/K $ whose underlying group $ N \times \rho(G_{L}) $ is abelian. By Theorem \ref{thm_main},  $ \B\OE $ is a free $ \OE[N \times \rho(G_{L})]^{G} $-module. Since $ E/L $ is unramified, we may apply Proposition \ref{prop_non_normal_descent} to conclude that $ \B $ is a free module over $ \pi( \OE [N \times \rho(G_{L})]^{G}) = \OL[N]^{G} $.
\end{proof}

\section{Abelian Extensions of Number Fields} \label{section_number_fields}

Let $ L/K $ be  a tame Galois extension of number fields with group $ G $ and $ \B $ an ambiguous fractional ideal of $ L $. In section \ref{section_introduction}, we noted that in this situation $ \B $ is locally, but not necessarily globally, free over $ \OK[G] $, and we discussed some of the applications of nonclassical Hopf-Galois structures to this problem. We might therefore desire number field analogues of the results established thus far. In showing that every commutative Hopf-Galois structure admitted by a tame Galois extension of $ p $-adic fields is induced (Proposition \ref{prop_tame_structures_induced}), we exploited the fact that in this case the Galois group is an extension (in the sense of group theory) of one cyclic group by another. This is no longer the case for tame Galois extensions of number fields, and so in order to prove analogous results we shall assume in this section that $ G $ is abelian. A consequence of this assumption is that if $ \p $ is a prime of $ \OK $ then the inertia groups of the primes lying above $ \p $ all coincide; thus we may speak of the inertia group of $ \p $. We write $ K_{\p} $ for the completion of $ K $ with respect to the absolute value associated  to $ \p $, and if $ A $ is a $ K $-algebra then we write $ A_{\p} = K_{\p} \otimes_{K} A $. Similarly, we let $ \OKp $ denote the valuation ring of $ K_{\p} $, and if $ M $ is an $ \OK $-module then we write $ M_{\p} = \OKp \otimes_{\OK} M $. We note that $ L_{\p} $ need not be a field, so the results of sections \ref{section_action_of_G0} and \ref{section_induced} cannot be applied verbatim. In this section we discuss appropriate modifications of the results established in those sections. 
\\ \\
First we establish an analogue of Theorem \ref{thm_G0_trivial_p_part}:

\begin{prop} \label{prop_global_G0_trivial_p_part}
Let $ \p $ be a prime of $ \OK $, $ p $ the prime number lying below $ \p $, and $ G_{0} $ the inertia group of $ \p $. If $ G_{0} $ acts trivially on the $ p $-part of $ N $ then $ \Bp $ is a free $\Lambda_{\p}^{G} $-module. 
\end{prop}
\begin{proof}
First, we note that the direct analogue of Proposition \ref{prop_free_iff_projective} holds: the ring $ \Lambda_{\p}^{G} = \OLp[N]^{G} $ is an $ \OKp $-order in the commutative separable $ K_{\p} $-algebra $ H_{\p} $, so it is a clean order. Since $ L $ is a free $ H $-module of rank one, $ L_{\p} $ is a free $ H_{\p} $-module of rank one, so $ \Bp $ is a free $ \Lambda_{\p}^{G} $-module if and only if it is projective. 
\\ \\
As in section \ref{section_action_of_G0}, we let $ \Gamma = \mathrm{Map}(G,\B) $; we then have $ \Gamma_{\p} = \mathrm{Map}(G,\Bp) $, and $ \Gamma_{\p}^{G} \cong \Bp $ via the analogue of isomorphism expressed in equation \eqref{eqn_GP_isomorphism}. Next, we state the appropriate analogy of Proposition \ref{prop_G_to_G0}. Let $ L_{0} $ be the fixed field of $ G_{0} $. Then $ \p $ is unramified in $ L_{0} $, and the extension of completed rings of integers $ {\mathfrak O}_{L_{0},\p} / \OKp $ is a Galois extension with group $ G $. Therefore by the method of the proof of Proposition \ref{prop_G_to_G0}, $ \Gamma_{\p}^{G} $ is a projective $ \Lambda_{\p}^{G} $-module if and only if $ \Gamma_{\p}^{G_{0}} $ is a projective $ \Lambda_{\p}^{G_{0}} $-module. 
\\ \\
As in section \ref{section_action_of_G0}, let $ T $ be the $ p $-part, and $ S $ the prime-to-$p$-part, of $ N $. If $ G_{0} $ acts trivially on $ T $ then the argument of Proposition \ref{prop_LambdaG0_group_ring} shows that $\Lambda_{\p}^{G_{0}} = \OLp[S]^{G_{0}}[T] $, and the argument of \cite[Proposition 5.6]{PJT_Towards} shows that $ \OLp[S]^{G_{0}} $ is the unique maximal order in the separable commutative $ K_{\p} $-algebra $ L_{\p}[S]^{G_{0}} $, in analogy with Proposition \ref{prop_OLSG0_maximal}. The construction of an element $ f \in \mathrm{Map}(G,\OLp)^{G} $ satisfying $ \theta_{T} \cdot f = 1 $, and the proof that multiplication by $ f $ is an $ \OLp[S]^{G_{0}} $-linear endomorphism of $ \Gamma^{G_{0}}_{\p} $, proceed exactly as in Propositions \ref{prop_element_f_analogous_to_x} and \ref{prop_mult_by_f}. This element $ f $ can then be used to extend an $ \OLp[S]^{G_{0}} $-homomorphism to a $ \Lambda_{\p}^{G_{0}} $-homomorphism, as in Proposition \ref{prop_Lambda_G_0_hom}. 
\\ \\
Finally, suppose that $ X $ is a free $ \Lambda_{\p}^{G_{0}} $-module and $ \varphi : X \rightarrow \Gamma^{G_{0}}_{\p} $ is a surjective $  \Lambda_{\p}^{G_{0}} $-homomorphism. Since $ \OLp[S]^{G_{0}} $ is the unique maximal order in $ L_{\p}[S]^{G_{0}} $, there exists an $ \OLp[S]^{G_{0}} $-homomorphism that splits $ \varphi $, and we may extend this to an $ \Lambda_{\p}^{G_{0}} $-homomorphism that splits $ \varphi $ as in the proof of Proposition \ref{prop_GammaG0_proj}. Therefore $ \Gamma^{G_{0}}_{\p} $ is a projective $ \Lambda_{\p}^{G_{0}} $-module, and so $ \Gamma^{G}_{\p} \cong \Bp $ is a free $ \Lambda_{\p}^{G} $-module.
\end{proof}

Now we have:

\begin{thm} \label{thm_global_abelian_tame_locally_free}
Let $ L/K $ be a tame abelian extension of number fields, $ H $ a commutative Hopf algebra giving a Hopf-Galois structure on $ L/K $, and $ \B $ an ambiguous fractional ideal of $ L $. Then $ \B $ is a locally free $ \OL[N]^{G} $-module. 
\end{thm}
\begin{proof}
Let $ \p $ be a prime of $ \OK $, $ p $ the prime number lying below $ \p $, and $ T $ the $ p $-part of $ N $. Let $ F = L^{T} $ and $ G_{F} = \Gal{L/F} $. Then $ G_{F} $ is the $ p $-part of $ G $, and the prime-to-$p$-part of $ G $, say $ C $, is a normal complement to $ G_{F} $ in $ G $. Therefore by \cite[Theorem 9]{Crespo_Induced} the Hopf-Galois structure given by $ H $ on $ L/K $ is induced from Hopf-Galois structures on $ F/K $ and $ L/F $, and by Proposition \ref{prop_induced_complement_trivial_action} the action of $ C $ on $ T $ is trivial. Since $ \p $ is tamely ramified in $ L $, the order of $ G_{0} $ is not divisible by $ p $, and so $ G_{0} $ is contained in $ C $. Therefore the action of $ G_{0} $ on $ T $ is trivial, and so by Proposition \ref{prop_global_G0_trivial_p_part} $ \Bp $ is a free $ \OLp[N]^{G} $-module. 
\end{proof}

\end{document}